\documentclass{llncs}
\usepackage{graphicx}
\usepackage{amssymb,boxedminipage,amsmath,comment}
\usepackage[dvipsnames]{xcolor}
\usepackage{tikz}

\newcommand{\NP}{{\sf NP}}

\newcommand{\ssi}{\subseteq_i}
\newcommand{\si}{\supseteq_i}

\newcommand{\radius}{{\sf radius}}
\newcommand{\diam}{{\sf diameter}}

\newtheorem{observation}{Observation}
\newtheorem{open}{Open Problem}

\newcommand{\dia}{\hfill{$\diamond$}}

\definecolor{nicered}{RGB}{204,0,0}
\definecolor{lightblue}{RGB}{153,204,255}

\newcommand{\defproblem}[3]{
 \vspace{1mm}
\noindent\fbox{
 \begin{minipage}{0.96\textwidth}
 \begin{tabular*}{\textwidth}{@{\extracolsep{\fill}}lr} #1 & \\ \end{tabular*}
 {\bf{Input:}} #2 \\
 {\bf{Question:}} #3
 \end{minipage}
 } 
 \vspace{1mm}
}

\title{On The Complexity of Matching Cut for\\ Graphs of Bounded Radius and $H$-Free Graphs}

\author{Felicia Lucke\inst{1} \and Dani\"el Paulusma\inst{2} \and Bernard Ries\inst{1}}

\institute{University of Fribourg, Department of Informatics, Fribourg, Switzerland,  \texttt{\{felicia.lucke,bernard.ries\}@unifr.ch}
\and
Durham University, Durham, UK, \texttt{daniel.paulusma@durham.ac.uk}}

\begin{document}
\maketitle
\setcounter{footnote}{0}

\begin{abstract}
For a connected graph $G=(V,E)$, a matching $M\subseteq E$  is a matching cut of~$G$ if $G-M$ is disconnected. It is known that for an integer~$d$, the corresponding decision problem {\sc Matching Cut} is polynomial-time solvable for graphs of diameter at most~$d$ if $d\leq 2$ and \NP-complete if $d\geq 3$. We prove the same dichotomy for graphs of bounded radius.  For a graph $H$, a graph is $H$-free if it does not contain $H$ as an induced subgraph. As a consequence of our result, we can solve {\sc Matching Cut} in polynomial time for $P_6$-free graphs, extending a recent result of Feghali for $P_5$-free graphs. We then extend our result to hold even for $(sP_3+P_6)$-free graphs for every $s\geq 0$ and initiate a complexity classification of {\sc Matching Cut} for $H$-free graphs.

\end{abstract}

\section{Introduction}\label{s-intro}

Let $G=(V,E)$ be an (undirected) connected graph. 
A subset of edges $M\subseteq E$ is a {\it matching} if no two edges in $M$ have a common end-vertex, whilst 
$M$ is an {\it edge cut} if $V$ can be partitioned into sets $B$ and $R$, such that $M$ consists of all the edges with one end-vertex in $B$ and the other one in $R$. We say that $M$ is a {\it matching cut} if $M$ is a matching that is also an edge cut; see
Fig.~\ref{f-matchcut_colouring} for an example.
Matching cuts have applications in number theory~\cite{Gr70}, graph drawing~\cite{PP01}, graph homomorphisms~\cite{GPS12}, edge labelings~\cite{ACGH12} and ILFI networks~\cite{FP82}.
The corresponding decision problem is defined as follows:

\medskip
\defproblem{{\sc Matching Cut}}
{a connected graph $G$.}
{does $G$ have a matching cut?} 

\medskip
\noindent
Chv{\'{a}}tal~\cite{Ch84} proved that {\sc Matching Cut} is \NP-complete. This led to an extensive study on the computational complexity of the problem restricted to special graph 
classes~\cite{Bo09,BJ08,CHLLP21,Fe,KL16,LL19,LR03,Mo89}.
We discuss some relevant results in this section and in Section~\ref{s-con}; see~\cite{CHLLP21} for a more detailed overview of known algorithmic results, including exact and parameterized algorithms~\cite{AKK17,AS21,KKL20,KL16}.

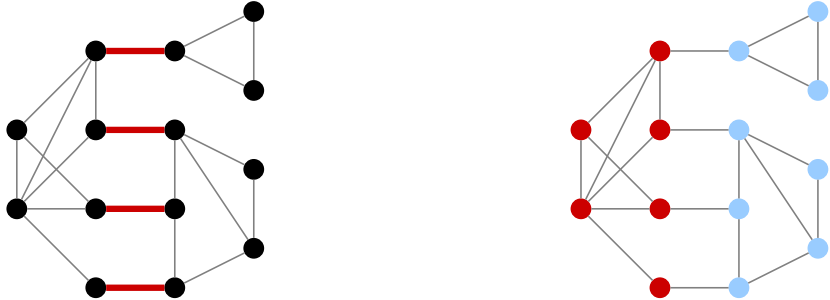
\begin{figure}
\begin{center}
\scalebox{1.5}{\begin{tikzpicture}
\tikzstyle{bvertex}=[thin,circle,inner sep=0.cm, minimum size=1.7mm, fill=lightblue, draw=lightblue]
    \tikzstyle{rvertex}=[thin,circle,inner sep=0.cm, minimum size=1.7mm, fill=nicered,draw=nicered]
    \tikzstyle{vertex}=[thin,circle,inner sep=0.cm, minimum size=1.7mm, fill=black, draw=black]
    \tikzstyle{hedge}=[thin, draw = gray]
    \tikzstyle{cutedge}=[ultra thick, draw=nicered]

 \def\k{0.7} 

	\node[vertex] (a1) at (0,\k/2){};
	\node[vertex] (a2) at (0,-\k/2){};
	\node[vertex] (a3) at (\k,1.5*\k){};
	\node[vertex] (a4) at (\k,\k/2){};
	\node[vertex] (a5) at (\k,-3*\k/2){};
	\node[vertex] (a5a) at (\k,-\k/2){};
	\node[vertex] (a6) at (2*\k,1.5*\k){};
	\node[vertex] (a7) at (2*\k,0.5*\k){};
	\node[vertex] (a8a) at (2*\k,-0.5*\k){};
	\node[vertex] (a8) at (2*\k,-1.5*\k){};
	\node[vertex] (a9) at (3*\k,2*\k){};
	\node[vertex] (a10) at (3*\k,1*\k){};
	\node[vertex] (a11) at (3*\k,0*\k){};
	\node[vertex] (a12) at (3*\k,-1*\k){};
	
	\draw[hedge](a1)--(a2);
	\draw[hedge](a1)--(a3);
	\draw[hedge](a1)--(a5a);
	\draw[hedge](a2)--(a3);
	\draw[hedge](a2)--(a4);
	\draw[hedge](a2)--(a5);
	\draw[hedge](a2)--(a5a);
	\draw[hedge](a3)--(a4);
	\draw[cutedge](a3)--(a6);
	\draw[cutedge](a4)--(a7);
	\draw[cutedge](a5)--(a8);
	\draw[cutedge](a5a)--(a8a);
	\draw[hedge](a6)--(a9);
	\draw[hedge](a6)--(a10);
	\draw[hedge](a7)--(a8a);
	\draw[hedge](a7)--(a11);
	\draw[hedge](a7)--(a12);
	\draw[hedge](a8)--(a12);
	\draw[hedge](a8a)--(a8);
	\draw[hedge](a9)--(a10);
	\draw[hedge](a11)--(a12);
	
	\begin{scope}[shift= {(5,0)}]
	\node[rvertex] (a1) at (0,\k/2){};
	\node[rvertex] (a2) at (0,-\k/2){};
	\node[rvertex] (a3) at (\k,1.5*\k){};
	\node[rvertex] (a4) at (\k,\k/2){};
	\node[rvertex] (a5) at (\k,-3*\k/2){};
	\node[rvertex] (a5a) at (\k,-\k/2){};
	\node[bvertex] (a6) at (2*\k,1.5*\k){};
	\node[bvertex] (a7) at (2*\k,0.5*\k){};
	\node[bvertex] (a8a) at (2*\k,-0.5*\k){};
	\node[bvertex] (a8) at (2*\k,-1.5*\k){};
	\node[bvertex] (a9) at (3*\k,2*\k){};
	\node[bvertex] (a10) at (3*\k,1*\k){};
	\node[bvertex] (a11) at (3*\k,0*\k){};
	\node[bvertex] (a12) at (3*\k,-1*\k){};
	
	\draw[hedge](a1)--(a2);
	\draw[hedge](a1)--(a3);
	\draw[hedge](a1)--(a5a);
	\draw[hedge](a2)--(a3);
	\draw[hedge](a2)--(a4);
	\draw[hedge](a2)--(a5);
	\draw[hedge](a2)--(a5a);
	\draw[hedge](a3)--(a4);
	\draw[hedge](a3)--(a6);
	\draw[hedge](a4)--(a7);
	\draw[hedge](a5)--(a8);
	\draw[hedge](a5a)--(a8a);
	\draw[hedge](a6)--(a9);
	\draw[hedge](a6)--(a10);
	\draw[hedge](a7)--(a8a);
	\draw[hedge](a7)--(a11);
	\draw[hedge](a7)--(a12);
	\draw[hedge](a8)--(a12);
	\draw[hedge](a8a)--(a8);
	\draw[hedge](a9)--(a10);
	\draw[hedge](a11)--(a12);
	
	\end{scope}
\end{tikzpicture}}
\caption{Left: an example of a matching cut, as indicated by the thick red edges. Right: the corresponding valid red-blue colouring, an equivalent definition given in Section~\ref{s-pre}.}\label{f-matchcut_colouring} 
\end{center}
\end{figure}

Let $G$ be a connected graph.
The {\it distance} between two vertices $u$ and $v$ in~$G$ is the {\it length} (number of edges) of a shortest path between $u$ and $v$ in~$G$. The {\it eccentricity} of a vertex $u$ is the maximum distance between $u$ and any other vertex of $G$. The {\it diameter} of $G$ is the maximum eccentricity over all vertices of~$G$.
Borowiecki and Jesse{-}J{\'{o}}zefczyk~\cite{BJ08} proved that {\sc Matching Cut} is polynomial-time solvable for graphs of diameter~$2$. Le and Le~\cite{LL19} gave a faster polynomial-time algorithm for graphs of diameter~$2$ and proved the following dichotomy.
            
\begin{theorem}[\cite{LL19}]\label{t-diameter}
For an integer~$d$, {\sc Matching Cut} for graphs of diameter at most~$d$ is polynomial-time solvable if $d\leq 2$ and \NP-complete if $d\geq 3$.
\end{theorem}

\noindent
Le and Le~\cite{LL19} also proved that {\sc Matching Cut} for bipartite graphs of diameter at most~$d$ is polynomial-time solvable if $d\leq 3$ and \NP-complete for $d\geq 4$. 
Another recent dichotomy is due to Chen et al.~\cite{CHLLP21}, who extended results of Le and Randerath~\cite{LR03} and proved that {\sc Matching Cut} for graphs of minimum degree~$\delta$ is polynomial-time solvable if $\delta=1$ and \NP-complete if $\delta\geq 2$ 
(note that the problem is trivial if $\delta=1$).

The {\it radius} of a connected graph~$G$ is closely related to the diameter; it is defined as the minimum eccentricity over all vertices of~$G$. It is readily seen that for every connected graph~$G$, 
\[\radius(G)\leq \diam(G)\leq 2\cdot\radius(G).\]
Complexity dichotomies for graphs of bounded radius have been studied in the literature; for example,
Mertzios and Spirakis~\cite{MS16} showed that {\sc $3$-Colouring} is \NP-complete for graphs of diameter~$3$ and radius~$2$, whilst {\sc $3$-Colouring} is trivial for graphs of radius~$1$.

\subsection*{Our Results}

We will prove the following dichotomy for general graphs of bounded radius, which strengthens the polynomial part of Theorem~\ref{t-diameter} (in order to see this, consider for example an arbitrary star and subdivide each of its edges once; the new graph has radius~$2$ but its diameter is~$4$).

\begin{theorem}\label{t-main}
For an integer~$r$, {\sc Matching Cut} for graphs of radius at most~$r$ is polynomial-time solvable if $r\leq 2$ and \NP-complete if $r\geq 3$.
\end{theorem}

\noindent
We prove Theorem~\ref{t-main} in Section~\ref{s-main} after giving some more terminology in Section~\ref{s-pre}. In Section~\ref{s-aux} we present some known results that we need as lemmas for proving Theorem~\ref{t-main}. In particular, we will use the reduction rules of Le and Le~\cite{LL19}, which they used in their polynomial-time algorithms for graphs of diameter~$2$ and bipartite graphs of diameter~$3$.

A graph $H$ is an {\it induced subgraph} of $G$ if $H$ can be obtained from $G$ after removing all vertices of $V(G)\setminus V(H)$. A graph~$G$ is {\it $H$-free} if $G$ does not contain an induced subgraph isomorphic to $H$.
We let $P_r$ denote the path on $r$ vertices.
Feghali~\cite{Fe} recently proved that {\sc Matching Cut} is polynomial-time solvable for $P_5$-free graphs and that there exists an integer~$r$, such that {\sc Matching Cut} is \NP-complete for $P_r$-free graphs. 
In a recent paper~\cite{LPRb}, we showed that the constant $r$ in~\cite{Fe} is equal to~$27$. In the same paper~\cite{LPRb} we proved that {\sc Matching Cut} is \NP-complete even for $(4P_5,P_{19})$-free graphs (by a slight modification of the construction from~\cite{Fe}).

As a consequence of the polynomial part of Theorem~\ref{t-main} we can show the following result.

\begin{corollary}\label{c-p6}
{\sc Matching Cut} is polynomial-time solvable for $P_6$-free graphs.
\end{corollary}

\noindent
We prove Corollary~\ref{c-p6} in Section~\ref{s-main} as well. 
In Section~\ref{s-extension} we prove that if {\sc Matching Cut} is polynomial-time solvable on a class of $H$-free graphs, then it is so on the class of $(P_3+H)$-free graphs (here, the graph $P_3+H$ denotes the disjoint union of the graphs $P_3$ and $H$). This means in particular that {\sc Matching Cut} is polynomial-time solvable even for $(sP_3+P_6)$-free graphs for every $s\geq 0$. 

In Section~\ref{s-class} we show some new hardness results of {\sc Matching Cut} for $H$-free graphs. In the same section, we also combine all our new results with known results to give a state-of-the-art summary of {\sc Matching Cut} for $H$-free graphs. We finish our paper with a number of open problems in Section~\ref{s-con}.

\section{Preliminaries}\label{s-pre}

Let $G=(V,E)$ be a graph.
For a vertex $u$, we let $N(u)=\{v\; |\; uv\in E\}$ denote the {\it neighbourhood} of $u$ in $G$.
Let $S\subseteq V$. The {\it neighbourhood} of $S$ in $G$ is the set $N(S)=\bigcup_{u\in S}N(u)\setminus S$. 
We let $G[S]$ denote the subgraph of $G$ {\it induced} by $S$, that is, $G[S]$ can be obtained from $G$ after deleting the vertices of $S$.
Moreover, $S$ is a {\it dominating} set of $G$ if every vertex of $V\setminus S$ has at least one neighbour in~$S$. 
In that case we also say that $G[S]$ {\it dominates} $G$.
The {\it domination number} of a graph $G$ is the size of a smallest dominating set of $G$.

Recall that we denote the path on $r$ vertices by $P_r$. We let $C_s$ denote the cycle on $s$ vertices.
A bipartite graph with non-empty partition classes $V_1$ and $V_2$ is {\it complete} if there exists an edge between 
every vertex of $V_1$ and every vertex of $V_2$. We let $K_{n_1,n_2}$ denote the complete bipartite graph with partition classes of size $n_1$ and $n_2$, respectively. The graph $K_{1,n_2}$ denotes the {\it star} on $n_2+1$ vertices.

We denote the {\it disjoint union} of two graphs $G_1$ and $G_2$ by $G_1+G_2=(V(G_1)\cup V(G_2),E(G_1)\cup E(G_2))$. We denote the disjoint union of $s$ copies of a graph $G$ by $sG$.

When constructing matching cuts we will often use an equivalent definition in terms of vertex colourings (see also~\cite{CHLLP21,Ch84,Fe}). Let $G=(V,E)$ be a connected graph. A {\it red-blue colouring} of $G$ colours every vertex of $G$ either red or blue. A red-blue colouring of $G$ is {\it valid} if every blue vertex has at most one red neighbour; every red vertex has at most one blue neighbour; and both colours red and blue are used at least once. We refer to Fig.~\ref{f-matchcut_colouring} for an example.

For a red-blue colouring, we let $R$ and $B$ denote the sets that consist of all vertices coloured red or blue, respectively (so $V=R\cup B$). We call $R$ the {\it red} set and $B$ the {\it blue} set of the red-blue colouring.
We let $R'$ consist of all vertices in $R$ that have a (unique) blue neighbour, and similarly, we let $B'$ consist of all vertices of $B$ that have a (unique) red neighbour. We call $R'$ the {\it red interface} and $B'$ the {\it blue interface} of the red-blue colouring.

From a matching cut $M$ of a connected graph $G$ we can construct a valid red-blue colouring by colouring the vertices of one connected component of $G-M$ red and all other vertices of $G$ blue. Similarly, from a valid red-blue colouring we can construct a matching cut by taking all edges with one end-vertex in the red interface and one end-vertex in the blue interface. Hence, we can make the following observation, which as mentioned is well known.

\begin{observation}\label{o}
A connected graph $G$ has a matching cut if and only if $G$ admits a valid red-blue colouring.
\end{observation}

\section{Auxiliary Results}\label{s-aux}

\noindent
In the remainder of our paper, we use Observation~\ref{o} and mainly search for valid red-blue colourings.
We need the following lemma, which has been used (implicitly) to prove other results for {\sc Matching Cut} as well, such as the result for $P_5$-free graphs~\cite{Fe}. We include a short proof for completeness.

\begin{lemma}\label{l-dom}
For every integer $g$, it is possible to find in $O(2^gn^{g+2})$ time a valid red-blue colouring (if it exists)
of an $n$-vertex graph with domination number~$g$.
\end{lemma}

\begin{proof}
Let $g\geq 1$ be an integer, and let $G$ be a graph with domination number at most $g$.
Hence, $G$ has a dominating set $D$ of size at most $g$. We consider all options of colouring the vertices of $D$ red or blue; note that this number is $2^{|D|}\leq 2^g$. For every red vertex of $D$ with no blue neighbour, we consider all $O(n)$ options of colouring at most one of its neighbours blue (and thus all of its other neighbours will be coloured red).
Similarly, for every blue vertex of $D$ with no red neighbour, we consider all $O(n)$ options of colouring at most one of its neighbours red (and thus all of its other neighbours will be coloured blue). Finally, for every red vertex in $D$ with already one blue neighbour in $D$, we colour all its yet uncoloured neighbours red. 
Similarly, for every blue vertex in $D$ with already one red neighbour in $D$, we colour all its yet uncoloured neighbours blue.

As $D$ is a dominating set, the above means that we guessed a red-blue colouring of the whole graph $G$. We can check in $O(n^2)$ time if a red-blue colouring is valid. Moreover, the total number of red-blue colourings that we must consider is $O(2^gn^g)$. 
\qed
\end{proof}

\noindent
Consider a red-blue colouring of a graph $G=(V,E)$.
A subset $S\subseteq V$ is \textit{monochromatic} if every vertex of $S$ has the same colour.
We need the following known lemma (see e.g.~\cite{Fe} which uses it implicitly); again we added a short proof for completeness.

\begin{lemma}\label{l-dom2}
Let $D$ be a dominating set of a connected graph $G$. It is possible to check in polynomial time if $G$ has a valid red-blue colouring in which $D$ is monochromatic.
\end{lemma}

\begin{proof}
Consider a valid red-blue colouring of $G$, in which $D$ is monochromatic.
Assume without loss of generality that every vertex of $D$ is coloured red.
Let $C$ be a connected component of $G-D$. If $C$ is not monochromatic, then $C$ has an edge $uv$ where $u$ is red and $v$ is blue. However, now $v$ has at least two red neighbours, namely $u$ and a neighbour in $D$ (such a neighbour exists, as $D$ is a dominating set of $G$). This is a contradiction. Hence, the vertex set of every connected component is monochromatic. Moreover, we may assume without loss of generality that the vertices of exactly one connected component of $G-D$ are coloured blue (else we can safely recolour the vertices of some connected component from blue to red). Hence, we can check all $O(n)$ options of choosing this unique blue connected component. For each option we check in polynomial time if the obtained red-blue colouring is valid.
\qed
\end{proof}

\begin{figure}
\begin{center}
\scalebox{1.2}{\begin{tikzpicture}
\tikzstyle{bvertex}=[thin,circle,inner sep=0.cm, minimum size=1.7mm, fill=lightblue, draw=lightblue]
    \tikzstyle{rvertex}=[thin,circle,inner sep=0.cm, minimum size=1.7mm, fill=nicered,draw=nicered]
    \tikzstyle{vertex}=[thin,circle,inner sep=0.cm, minimum size=1.7mm, fill=black, draw=black]
    \tikzstyle{hedge}=[thin, draw = gray]
    \tikzstyle{cutedge}=[ultra thick, draw=OrangeRed]
    \tikzstyle{rahmen1}=[rounded corners = 5pt,draw, dashed, minimum width = 15pt, minimum height = 35pt]
    \tikzstyle{rahmen2}=[rounded corners = 5pt,draw, , minimum width = 43pt, minimum height = 60pt]
    \tikzstyle{rahmen3}=[rounded corners = 5pt,draw, dashed, minimum width = 15pt, minimum height = 35pt]
    \tikzstyle{rahmen4}=[rounded corners = 5pt,draw, , minimum width = 43pt, minimum height = 58.5pt]

 \def\k{0.8} 

\node[rvertex] (a1) at (0,\k/2){};
	\node[rvertex] (a2) at (0,-\k/2){};
	\node[rvertex] (a3) at (\k,1.5*\k){};
	\node[rvertex] (a4) at (\k,\k/2){};
	\node[rvertex] (a5) at (\k,-3*\k/2){};
	\node[rvertex] (a5a) at (\k,-\k/2){};
	\node[bvertex] (a6) at (2*\k,1.5*\k){};
	\node[bvertex] (a7) at (2*\k,0.5*\k){};
	\node[bvertex] (a8a) at (2*\k,-0.5*\k){};
	\node[bvertex] (a8) at (2*\k,-1.5*\k){};
	\node[bvertex] (a9) at (3*\k,2*\k){};
	\node[bvertex] (a10) at (3*\k,1*\k){};
	\node[bvertex] (a11) at (3*\k,0*\k){};
	\node[bvertex] (a12) at (3*\k,-1*\k){};
	
	\draw[hedge](a1)--(a2);
	\draw[hedge](a1)--(a3);
	\draw[hedge](a1)--(a5a);
	\draw[hedge](a2)--(a3);
	\draw[hedge](a2)--(a4);
	\draw[hedge](a2)--(a5);
	\draw[hedge](a2)--(a5a);
	\draw[hedge](a3)--(a4);
	\draw[hedge](a3)--(a6);
	\draw[hedge](a4)--(a7);
	\draw[hedge](a5)--(a8);
	\draw[hedge](a5a)--(a8a);
	\draw[hedge](a6)--(a9);
	\draw[hedge](a6)--(a10);
	\draw[hedge](a7)--(a8a);
	\draw[hedge](a7)--(a11);
	\draw[hedge](a7)--(a12);
	\draw[hedge](a8)--(a12);
	\draw[hedge](a8a)--(a8);
	\draw[hedge](a9)--(a10);
	\draw[hedge](a11)--(a12);
	
	\node[rahmen1] at (\k,0){};
	\node[] at(\k,0){$S$};
	\node[rahmen2] at (\k/2,0.35){};
	\node[] at(0,1){$X$};
	\node[rahmen3] at (2*\k,0){};
	\node[] at(15*\k/8,0){$T$};
	\node[rahmen4] at (2.5*\k,-0.35){};
	\node[] at(3*\k,-1.1){$Y$};

\end{tikzpicture}}
\caption{An example of a graph $G$ with a valid red-blue $(S,T,X,Y)$-colouring.}\label{f-STXYcolouring} 
\end{center}
\end{figure}
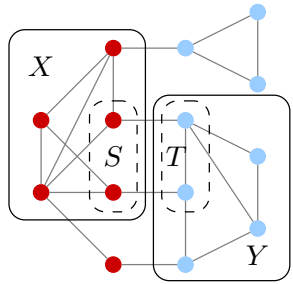

\noindent
Let $G=(V,E)$ be a connected graph and $S,T,X,Y\subseteq V$ be four non-empty sets with $S\subseteq X$, $T\subseteq Y$ and $X\cap Y=\emptyset$. A  {\it red-blue $(S,T,X,Y)$-colouring} of $G$
is a red-blue colouring with a red set containing $X$; a blue set containing $Y$; a red interface containing $S$ and a blue interface containing $T$. Note that $V\setminus (X\cup Y)$ might be non-empty. Moreover, the red and blue interfaces may also contain vertices not in $S$ and $T$, respectively. We refer to Fig.~\ref{f-STXYcolouring} for an example.

Now, let $S'$ and $T'$ be two non-empty subsets of $V$ with $S'\cap T'=\emptyset$ such that every vertex of $S'$ is adjacent to exactly one vertex of $T'$, and vice versa. We call $(S',T')$ a {\it starting pair} of $G$. 

For a starting pair, Le and Le~\cite{LL19} define five propagation rules that we give below -- in our terminology -- as rules R1--R3. The goal of these rules is to extend $S'$ and $T'$ by finding as many vertices as possible whose colour must be either red or blue in any valid red-blue colouring with a red interface containing $S'$ and a blue interface containing $T'$.  
We will place any newly found red vertices in a set $X$ that already contains $S'$ and any newly found blue vertices in a set $Y$ that already contains $T'$. We let $S$ be the subset of $X$ consisting of red vertices with a blue neighbour in $Y$, and we let $T$  be the subset of $Y$ consisting of blue vertices with a red neighbour in $X$. So it holds that $S'\subseteq S\subseteq X$ and $T'\subseteq T\subseteq Y$. The vertices of~$S$ will belong to the red interface and the vertices of $T$ will belong to the blue interface of the valid red-blue colouring that we are trying to construct. Note that the vertices of $X\setminus S$ and $Y\setminus T$ may or may not belong to the red or blue interface (as this depends on the colour of the uncoloured vertices in $V\setminus (X\cup Y)$, which we still have to determine).

We now state the propagation rules R1--R3.
Initially we set $X:=S:=S'$ and $Y:=T:=T'$.  As mentioned, R1--R3 try to put vertices of $Z=V\setminus (X\cup Y)$ into one of the sets $S$, $X\setminus S$, $T$ or $Y\setminus T$. They are defined as follows (see Fig.~\ref{f-starting_pair} for an example):

\begin{itemize}
\item [{\bf R1.}] Return {\tt no} (that is, $G$ has no valid red-blue $(S',T',S',T')$-colouring) if  a vertex $v\in Z$ is

\begin{itemize}
\item adjacent to a vertex in $S$ and to a vertex in $T$, or
\item adjacent to a vertex in $S$ and to two vertices in $Y\setminus T$, or
\item adjacent to a vertex in $T$ and to two vertices in $X\setminus S$, or
\item adjacent to two vertices in $X\setminus S$ and to two vertices in $Y\setminus T$.\\[-10pt]
\end{itemize}

\item [{\bf R2.}]  Assume $v\in Z$ and R1 does not apply. If $v$ is adjacent to a vertex in $S$ or to two vertices of $X\setminus S$, then  move $v$ from $Z$ to $X$. If moreover $v$ is adjacent to a (unique) vertex $w$ in $Y$, then also add $v$ to $S$ and $w$ to $T$.\\[-10pt]
\item [{\bf R3}.] Assume $v\in Z$ and R1 does not apply. If $v$ is adjacent to a vertex in $T$ or to two vertices of $Y\setminus T$, then move $v$ from $Z$ to $Y$. If moreover $v$ is adjacent to a (unique) vertex $w$ in $X$, then also add $v$ to $T$ and $w$ to $S$.
\end{itemize}

\begin{figure}
\begin{center}
\scalebox{1.2}{\begin{tikzpicture}
%starting pair

\tikzstyle{bvertex}=[thin,circle,inner sep=0.cm, minimum size=1.7mm, fill=lightblue, draw=lightblue]
    \tikzstyle{rvertex}=[thin,circle,inner sep=0.cm, minimum size=1.7mm, fill=nicered,draw=nicered]
    \tikzstyle{b1vertex}=[thick,circle,inner sep=0.cm, minimum size=1.7mm, fill=lightblue, draw=black]
    \tikzstyle{r1vertex}=[thick,circle,inner sep=0.cm, minimum size=1.7mm, fill=nicered,draw=black]
    \tikzstyle{vertex}=[thin,circle,inner sep=0.cm, minimum size=1.7mm, fill=none, draw=black]
    \tikzstyle{hedge}=[thin, draw = gray]
    \tikzstyle{cutedge}=[ultra thick, draw=OrangeRed]
    \tikzstyle{rahmen1}=[rounded corners = 5pt,draw, dashed, minimum width = 15pt, minimum height = 35pt]
    \tikzstyle{rahmen2}=[rounded corners = 5pt,draw, , minimum width = 43pt, minimum height = 60pt]
    \tikzstyle{rahmen3}=[rounded corners = 5pt,draw, dashed, minimum width = 15pt, minimum height = 35pt]
    \tikzstyle{rahmen4}=[rounded corners = 5pt,draw, , minimum width = 43pt, minimum height = 58.5pt]

 \def\k{0.8} 

\node[rvertex] (a1) at (0,\k/2){};
	\node[rvertex] (a2) at (0,-\k/2){};
	\node[rvertex] (a3) at (\k,1.5*\k){};
	\node[r1vertex] (a4) at (\k,\k/2){};
	\node[vertex] (a5) at (\k,-3*\k/2){};
	\node[rvertex] (a5a) at (\k,-\k/2){};
	\node[vertex] (a6) at (2*\k,1.5*\k){};
	\node[b1vertex] (a7) at (2*\k,0.5*\k){};
	\node[bvertex] (a8a) at (2*\k,-0.5*\k){};
	\node[bvertex] (a8) at (2*\k,-1.5*\k){};
	\node[vertex] (a9) at (3*\k,2*\k){};
	\node[vertex] (a10) at (3*\k,1*\k){};
	\node[bvertex] (a11) at (3*\k,0*\k){};
	\node[bvertex] (a12) at (3*\k,-1*\k){};
	
	\draw[hedge](a1)--(a2);
	\draw[hedge](a1)--(a3);
	\draw[hedge](a1)--(a5a);
	\draw[hedge](a2)--(a3);
	\draw[hedge](a2)--(a4);
	\draw[hedge](a2)--(a5);
	\draw[hedge](a2)--(a5a);
	\draw[hedge](a3)--(a4);
	\draw[hedge](a3)--(a6);
	\draw[hedge](a4)--(a7);
	\draw[hedge](a5)--(a8);
	\draw[hedge](a5a)--(a8a);
	\draw[hedge](a6)--(a9);
	\draw[hedge](a6)--(a10);
	\draw[hedge](a7)--(a8a);
	\draw[hedge](a7)--(a11);
	\draw[hedge](a7)--(a12);
	\draw[hedge](a8)--(a12);
	\draw[hedge](a8a)--(a8);
	\draw[hedge](a9)--(a10);
	\draw[hedge](a11)--(a12);
	
	\node[rahmen1] at (\k,0){};
	\node[] at(\k,0){$S$};
	\node[rahmen2] at (\k/2,0.35){};
	\node[] at(0,1){$X$};
	\node[rahmen3] at (2*\k,0){};
	\node[] at(15*\k/8,0){$T$};
	\node[rahmen4] at (2.5*\k,-0.35){};
	\node[] at(3*\k,-1.1){$Y$};

\end{tikzpicture}}
\caption{An example of an application of rules R1--R3 that results in a 4-tuple $(S,T,X,Y)$. The set $S'$ consists of the thick red vertex and the set $T'$ consists of the thick blue vertex. Note that $S'\subseteq S\subseteq X$ and $T'\subseteq T\subseteq Y$, and that every valid red-blue $(S',T',S',T')$-colouring is a valid red-blue $(S,T,X,Y)$-colouring.}\label{f-starting_pair} 
\end{center}
\end{figure}
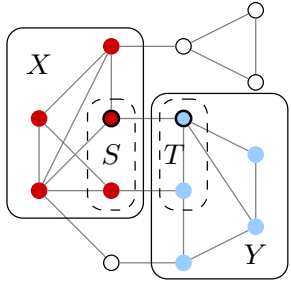

\noindent
Le and Le~\cite{LL19} proved the following two lemmas. The first lemma
shows that rules R1--R3 are safe and 
is not difficult to verify, whereas 
the second lemma is proven by a reduction to {\sc $2$-Satisfiability}. We slightly changed the formulation of their lemma so that it can be applied to the case where the graph $G$ may also have a valid red-blue colouring in which not every connected component of $G-(X\cup Y)$ is monochromatic.

\begin{lemma}[\cite{LL19}]\label{l-ll1}
Let $G$ be a graph with a starting pair $(S',T')$. Assume that exhaustively applying rules R1--R3 did not lead to a no-answer but to a $4$-tuple $(S,T,X,Y)$. The following holds:
\begin{itemize}
\item [(i)] $S'\subseteq S\subseteq X$ and $T'\subseteq T\subseteq Y$ and $X\cap Y=\emptyset$,\\[-10pt]
\item [(ii)] $G$ has a valid red-blue $(S',T',S',T')$-colouring  if and only if $G$ has a valid red-blue $(S,T,X,Y)$-colouring (note that the backward implication holds by definition), and\\[-10pt]
\item [(iii)] every vertex in $V\setminus (X\cup Y)$ has no neighbour in $S\cup T$; at most one neighbour in $X\setminus S$ and at most one neighbour in $Y\setminus T$.
\end{itemize}
Moreover, the $4$-tuple $(S,T,X,Y)$ can be obtained in polynomial time.
\end{lemma}

\begin{lemma}[\cite{LL19}]\label{l-ll2}
Let $G$ be a graph with a starting pair $(S',T')$. Assume that exhaustively applying rules R1--R3 did not lead to a no-answer but to a $4$-tuple $(S,T,X,Y)$. It can be decided in $O(mn)$ time if $G$ has a valid red-blue $(S,T,X,Y)$-colouring (or equivalently, a valid red-blue $(S',T',S',T')$-colouring) in which every connected component of $G-(X\cup Y)$ is monochromatic.
\end{lemma}

\noindent
Finally, we need one more result from the literature (which has been strengthened in~\cite{CS16}).

\begin{theorem}[\cite{HP10}]\label{t-hp}
A graph $G=(V,E)$ on $n$ vertices is $P_6$-free if and only if each connected induced subgraph of $G$ 
contains a dominating induced $C_6$ or 
a dominating (not necessarily induced) complete bipartite graph.
Moreover, we can find such a dominating subgraph of $G$ in $O(n^3)$ time.
\end{theorem}

\section{The Proofs of Theorem~\ref{t-main} and Corollary~\ref{c-p6}}\label{s-main}

We first prove Theorem~\ref{t-main}, which we restate below.

\medskip
\noindent
{\bf Theorem~\ref{t-main} (restated).}
{\it  For an integer~$r$, {\sc Matching Cut} is polynomial-time solvable for graphs of radius at most $r$ if $r\leq 2$ and \NP-complete for graphs of radius at most $r$ if $r\geq 3$.}

\begin{proof}
The case where $r\geq 3$ follows from Theorem~\ref{t-diameter} after observing that the class of graphs of diameter at most~$3$ is contained in the class of graphs of radius at most~$3$. 
So assume now that $r\leq 2$. 
Let $G$ be a graph of radius at most $r$.
If $r=1$, then $G$ has a vertex that is adjacent to all other vertices. In this case $G$ has a matching cut if and only if $G$ has a vertex of degree~$1$; we can check the latter condition in polynomial time.
From now on, assume that $r=2$. Then $G$ has a dominating star $H$, say $H$ has centre $u$ and leaves $v_1,\ldots,v_s$ for some $s\geq 1$. By Observation~\ref{o} it suffices to check if $G$ has a valid red-blue colouring.

We first check if $G$ has a valid red-blue colouring in which $V(H)$ is monochromatic. By Lemma~\ref{l-dom2} this can be done in polynomial time. Suppose we find no such red-blue colouring. Then we may assume without loss of generality that a valid red-blue colouring of $G$ (if it exists) colours $u$ red and exactly one of $v_1,\ldots,v_s$ blue.
That is, $G$ has a valid red-blue colouring if and only if $G$ has a valid red-blue $(\{u\},\{v_i\},\{u\},\{v_i\})$-colouring for some $i\in \{1,\ldots,s\}$.
We consider all $O(n)$ options of choosing which $v_i$ is coloured blue. 

For each option we do as follows. Let $v_i$ be the vertex of $v_1,\ldots,v_s$ that we coloured blue. We define
the starting pair $(S',T')$ with $S'=\{u\}$ and $T'=\{v_i\}$ and apply rules R1--R3 exhaustively. The latter takes polynomial time by Lemma~\ref{l-ll1}. If this exhaustive application leads to a no-answer, then by Lemma~\ref{l-ll1} we may discard the option. Suppose we obtain a $4$-tuple $(S,T,X,Y)$. By again applying Lemma~\ref{l-ll1}, we find that
$G$ has a valid red-blue $(\{u\},\{v_i\},\{u\},\{v_i\})$-colouring if and only if $G$ has a valid red-blue $(S,T,X,Y)$-colouring. By R2 and the fact that
$u\in S'\subseteq S$ we find that $\{v_1,\ldots,v_s\}\setminus \{v_i\}$ belongs to~$X$.

Suppose that $G$ has a valid red-blue $(S,T,X,Y)$-colouring~$c$ such that $G-(X\cup Y)$ has a connected component~$D$ that is not monochromatic. Then $D$ must contain an edge $uv$, where $u$ is coloured red and $v$ is coloured blue. 
Note that $v$ cannot be adjacent to $v_i$, as otherwise $v$ would have been in $Y$ by R3 (since $v_i\in T'\subseteq T$). As $H$ is dominating, this means that $v$ must be adjacent to a vertex $w\in V(H)\setminus \{v_i\}=\{u,v_1,\ldots,v_s\}\setminus \{v_i\}$. 
As $u\in S'\subseteq S\subseteq X$ and $\{v_1,\ldots,v_s\}\setminus \{v_i\}\subseteq X$, we find that $w\in X$ by R2 and thus will be coloured red. However, now $v$ being coloured blue is adjacent to two red vertices (namely $u$ and $w$), contradicting the validity of $c$. 

From the above we conclude that for any valid $(S,T,X,Y)$-colouring (if it exists), every connected component $G-(X,Y)$ is monochromatic. Hence, we can apply Lemma~\ref{l-ll2} to find in polynomial time whether or not $G$ has a valid red-blue $(S,T,X,Y)$-colouring, or equivalently, if $G$ has a valid red-blue  $(\{u\},\{v_i\},\{u\},\{v_i\})$-colouring. 

The correctness of our algorithm follows from the above arguments.
As we branch $O(n)$ times and each branch takes polynomial time to process, the total running time of our algorithm is polynomial. \qed
\end{proof}

\noindent
We now prove Corollary~\ref{c-p6}, which we restate below.

\medskip
\noindent
{\bf Corollary~\ref{c-p6} (restated).}
{\it {\sc Matching Cut} is polynomial-time solvable for $P_6$-free graphs.}

\begin{proof}
Let $G$ be a connected $P_6$-free graph. By Theorem~\ref{t-hp}, we find that $G$ has a  dominating induced $C_6$ or 
a dominating (not necessarily induced) complete bipartite graph $K_{r,s}$. By Observation~\ref{o} it suffices to check if $G$ has a valid red-blue colouring.

If $G$ has a dominating induced $C_6$, then $G$ has domination number at most~$6$. In that case we apply Lemma~\ref{l-dom} to find in polynomial time if $G$ has a valid red-blue colouring.
Suppose that $G$ has a dominating complete bipartite graph $H$ with partition classes $\{u_1,\ldots,u_r\}$ and
$\{v_1,\ldots,v_s\}$. We may assume without loss of generality that $r\leq s$.

If $r\geq 2$ and $s\geq 3$, then it is readily seen that applying rules R1--R3 on any starting pair $(\{u_i\},\{v_j\})$ yields a no-answer. Hence, $V(H)$ is monochromatic for any valid red-blue colouring of $G$. 
This means that we can check in polynomial time by Lemma~\ref{l-dom2} if $G$ has a valid red-blue colouring.

Now assume that $r=1$ or $s\leq 2$. In the first case, $G$ has a (not necessarily induced) dominating star and thus $G$ has radius~2, and we apply Theorem~\ref{t-main}. In the second case, $r\leq s\leq 2$, and thus
$G$ has domination number at most~$4$, and we apply Lemma~\ref{l-dom}. Hence, in both cases, we find in polynomial time whether or not $G$ has a valid red-blue colouring. \qed
\end{proof}

\section{Extending Corollary~\ref{c-p6}}\label{s-extension}

In this section we slightly generalize the framework of Le and Le~\cite{LL19} in order to obtain new algorithms for {\sc Matching Cut} on $H$-free graphs. 

First we slightly generalize the definition of a starting pair $(S',T')$ of a graph~$G$. 
We still let $S'$ and $T'$ be two non-empty subsets of $V$ with $S'\cap T'=\emptyset$. However, now we only require that every vertex of $S'$ is adjacent to at most one vertex of $T'$, and vice versa, whilst at least one vertex of $S'$ must be adjacent to a vertex of~$T'$. 
We let $S''\subseteq S'$ consists of those vertices of $S'$ that have exactly one neighbour in $T'$.
Similarly, we let $T''$ consists of those vertices of $T'$ that  have exactly one neighbour in $S'$.
We call $(S',T')$ a {\it generalized starting pair} of $G$ with \textit{core} $(S'',T'')$. Note that by definition, $|S''|=|T''|\geq 1$. See Fig.~\ref{f-gen_starting_pair} for an example.

When we apply rules R1--R3, we first initiate by setting $S:=S''; X:=S'$; $T:=T''$; and $Y:=T'$; see also Fig.~\ref{f-gen_starting_pair}.
The following two lemmas can be readily checked by mimicking the proofs of Lemmas~\ref{l-ll1} and~\ref{l-ll2} given in~\cite{LL19}.

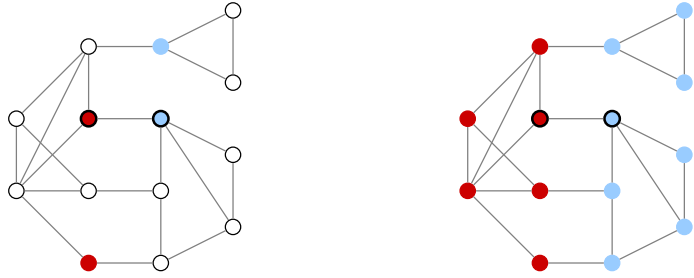
\begin{figure}[t]
\begin{center}
\scalebox{1.2}{\begin{tikzpicture}
%generalized starting pair

\tikzstyle{bvertex}=[thin,circle,inner sep=0.cm, minimum size=1.7mm, fill=lightblue, draw=lightblue]
    \tikzstyle{rvertex}=[thin,circle,inner sep=0.cm, minimum size=1.7mm, fill=nicered,draw=nicered]
    \tikzstyle{b1vertex}=[thick,circle,inner sep=0.cm, minimum size=1.7mm, fill=lightblue, draw=black]
    \tikzstyle{r1vertex}=[thick,circle,inner sep=0.cm, minimum size=1.7mm, fill=nicered,draw=black]
    \tikzstyle{vertex}=[thin,circle,inner sep=0.cm, minimum size=1.7mm, fill=none, draw=black]
    \tikzstyle{hedge}=[thin, draw = gray]
    \tikzstyle{cutedge}=[ultra thick, draw=OrangeRed]
    \tikzstyle{rahmen1}=[rounded corners = 5pt,draw, dashed, minimum width = 15pt, minimum height = 35pt]
    \tikzstyle{rahmen2}=[rounded corners = 5pt,draw, , minimum width = 43pt, minimum height = 60pt]
    \tikzstyle{rahmen3}=[rounded corners = 5pt,draw, dashed, minimum width = 15pt, minimum height = 35pt]
    \tikzstyle{rahmen4}=[rounded corners = 5pt,draw, , minimum width = 43pt, minimum height = 58.5pt]

 \def\k{0.8} 

	\node[vertex] (a1) at (0,\k/2){};
	\node[vertex] (a2) at (0,-\k/2){};
	\node[vertex] (a3) at (\k,1.5*\k){};
	\node[r1vertex] (a4) at (\k,\k/2){};
	\node[rvertex] (a5) at (\k,-3*\k/2){};
	\node[vertex] (a5a) at (\k,-\k/2){};
	\node[bvertex] (a6) at (2*\k,1.5*\k){};
	\node[b1vertex] (a7) at (2*\k,0.5*\k){};
	\node[vertex] (a8a) at (2*\k,-0.5*\k){};
	\node[vertex] (a8) at (2*\k,-1.5*\k){};
	\node[vertex] (a9) at (3*\k,2*\k){};
	\node[vertex] (a10) at (3*\k,1*\k){};
	\node[vertex] (a11) at (3*\k,0*\k){};
	\node[vertex] (a12) at (3*\k,-1*\k){};
	
	\draw[hedge](a1)--(a2);
	\draw[hedge](a1)--(a3);
	\draw[hedge](a1)--(a5a);
	\draw[hedge](a2)--(a3);
	\draw[hedge](a2)--(a4);
	\draw[hedge](a2)--(a5);
	\draw[hedge](a2)--(a5a);
	\draw[hedge](a3)--(a4);
	\draw[hedge](a3)--(a6);
	\draw[hedge](a4)--(a7);
	\draw[hedge](a5)--(a8);
	\draw[hedge](a5a)--(a8a);
	\draw[hedge](a6)--(a9);
	\draw[hedge](a6)--(a10);
	\draw[hedge](a7)--(a8a);
	\draw[hedge](a7)--(a11);
	\draw[hedge](a7)--(a12);
	\draw[hedge](a8)--(a12);
	\draw[hedge](a8a)--(a8);
	\draw[hedge](a9)--(a10);
	\draw[hedge](a11)--(a12);
	
	\begin{scope}[shift = ({5,0})]
		\node[rvertex] (a1) at (0,\k/2){};
	\node[rvertex] (a2) at (0,-\k/2){};
	\node[rvertex] (a3) at (\k,1.5*\k){};
	\node[r1vertex] (a4) at (\k,\k/2){};
	\node[rvertex] (a5) at (\k,-3*\k/2){};
	\node[rvertex] (a5a) at (\k,-\k/2){};
	\node[bvertex] (a6) at (2*\k,1.5*\k){};
	\node[b1vertex] (a7) at (2*\k,0.5*\k){};
	\node[bvertex] (a8a) at (2*\k,-0.5*\k){};
	\node[bvertex] (a8) at (2*\k,-1.5*\k){};
	\node[bvertex] (a9) at (3*\k,2*\k){};
	\node[bvertex] (a10) at (3*\k,1*\k){};
	\node[bvertex] (a11) at (3*\k,0*\k){};
	\node[bvertex] (a12) at (3*\k,-1*\k){};
	
	\draw[hedge](a1)--(a2);
	\draw[hedge](a1)--(a3);
	\draw[hedge](a1)--(a5a);
	\draw[hedge](a2)--(a3);
	\draw[hedge](a2)--(a4);
	\draw[hedge](a2)--(a5);
	\draw[hedge](a2)--(a5a);
	\draw[hedge](a3)--(a4);
	\draw[hedge](a3)--(a6);
	\draw[hedge](a4)--(a7);
	\draw[hedge](a5)--(a8);
	\draw[hedge](a5a)--(a8a);
	\draw[hedge](a6)--(a9);
	\draw[hedge](a6)--(a10);
	\draw[hedge](a7)--(a8a);
	\draw[hedge](a7)--(a11);
	\draw[hedge](a7)--(a12);
	\draw[hedge](a8)--(a12);
	\draw[hedge](a8a)--(a8);
	\draw[hedge](a9)--(a10);
	\draw[hedge](a11)--(a12);

	\end{scope}
	
%	\node[rahmen1] at (\k,0){};
%	\node[] at(\k,0){$S$};
%	\node[rahmen2] at (\k/2,0.35){};
%	\node[] at(0,1){$X$};
%	\node[rahmen3] at (2*\k,0){};
%	\node[] at(2*\k,0){$T$};
%	\node[rahmen4] at (2.5*\k,-0.35){};
%	\node[] at(3*\k,-1.1){$Y$};

\end{tikzpicture}}
\caption{Left: an example of a generalized starting pair $(S',T')$ with core $(S'',T'')$, where
$S'$ consists of the two red vertices, $S''$ consists of the thick red vertex, $T'$ consists of the two blue vertices and $T''$ consists of the thick blue vertex. Right: the application of rules R1-R3 on $(S',T')$.
Note that the resulting four tuple $(S,T,X,Y)$ immediately results in a valid red-blue colouring. Hence, having some  vertices in $S'\setminus S''$ and $T'\setminus T$, which are adjacent to a vertex with an opposite colour, can help significantly.}\label{f-gen_starting_pair} 
\end{center}
\end{figure}

\begin{lemma}\label{l-ll1b}
Let $G$ be a graph with a generalized starting pair $(S',T')$ with core $(S'',T'')$. Assume that exhaustively applying rules R1--R3 did not lead to a no-answer but to a $4$-tuple $(S,T,X,Y)$. The following holds:
\begin{itemize}
\item [(i)] $S''\subseteq S\subseteq X$; $S'\subseteq X$; $T''\subseteq T\subseteq Y$; $T'\subseteq Y$; and $X\cap Y=\emptyset$,\\[-10pt]
\item [(ii)] $G$ has a valid red-blue $(S'',T'',S',T')$-colouring  if and only if $G$ has a valid red-blue $(S,T,X,Y)$-colouring, and\\[-10pt]
\item [(iii)] every vertex in $V\setminus (X\cup Y)$ has no neighbour in $S\cup T$; at most one neighbour in $X\setminus S$ and at most one neighbour in $Y\setminus T$.
\end{itemize}
Moreover, the $4$-tuple $(S,T,X,Y)$ can be obtained in polynomial time.
\end{lemma}

\begin{lemma}\label{l-ll2b}
Let $G$ be a graph with a generalized starting pair $(S',T')$ with core $(S'',T'')$. Assume that exhaustively applying rules R1--R3 did not lead to a no-answer but to a $4$-tuple $(S,T,X,Y)$. It can be decided in $O(mn)$ time if $G$ has a valid red-blue $(S,T,X,Y)$-colouring (or equivalently, a valid red-blue $(S'',T'',S',T')$-colouring) in which every connected component of $G-(X\cup Y)$ is monochromatic.
\end{lemma}

\noindent
We are now ready to prove the main result of this section. 

\begin{theorem}\label{t-h}
Let $H$ be a graph. If {\sc Matching Cut} is polynomial-time solvable for $H$-free graphs, then it is so for $(H+P_3)$-free graphs.
\end{theorem}

\begin{proof}
Assume that {\sc Matching Cut} is polynomial-time solvable for $H$-free graphs. Let $G$ be a connected $(H+P_3)$-free graph with $n$ vertices and $m$ edges. We first check if $G$ has a matching cut of size at most~$2$. We can do this in polynomial time by considering all $O(m^2)$ options of choosing two edges. From now on we assume that $G$ has no matching cut of size at most~$2$; in particular this implies that $G$ has no vertex of degree~$1$.

We may also assume that $G$ has an induced subgraph $G'$ that is isomorphic to $H$; else we are done by our assumption. Let $G^*$ be the graph obtained from $G$ after removing every vertex of $V(G')\cup N(V(G'))$. 
As $G'$ is isomorphic to $H$ and $G$ is $(H+P_3)$-free, $G^*$ is $P_3$-free.

We continue as follows. We first consider all $O(m)$ options of choosing an edge from $E(G)$, one of whose end-vertices we colour red and the other one blue.
Afterwards, for each (uncoloured) vertex in $G'$ we consider all options of colouring it either red or blue. As $G'$ is isomorphic to $H$, there are $2^{|V(H)|}$ options of doing this. As $H$ is a fixed graph, this is a constant number.
Now, for every red vertex~$u$ of $G'$ with no blue neighbour, we consider all $O(n)$ options of colouring at most one of its neighbours blue (and thus all other not yet coloured neighbours of $u$ will be coloured red).
Similarly, for every blue vertex $v$ of $G'$ with no red neighbour, we consider all $O(n)$ options of colouring at most one of its neighbours red (and thus all other neighbours of $v$ will be coloured blue). Note that afterwards each vertex of $V(G')\cup N(V(G'))$ is either coloured red or blue. 

There are  $O(m2^{|V(H)|}n^{|V(H)|})$ options in total of colouring the end-vertices of an edge in $G$ and the vertices of $G'$. In each option, we have at least one red vertex and at least one blue vertex. We now consider the options one by one. 

Consider an option as described above.
In particular, let $e=uv$ be the chosen edge whose end-vertices we coloured differently, say we coloured $u$ red and $v$ blue.
We first check in polynomial time if every red vertex in this option has at most one blue neighbour and if every blue vertex has at most one red neighbour. If one of these two conditions does not hold, we discard the option. 
Now let $S'$ consist 
of $u$ and all red vertices of $V(G')\cup N(V(G'))$, and let $T'$ consists of $v$ and all blue vertices of $V(G')\cup N(V(G'))$. We let $S''\subseteq S'$ consist of all red vertices that have (exactly) one blue neighbour, and we let $T''\subseteq T'$ consist of all blue vertices that have (exactly) one red neighbour. By construction (recall that we started with picking an edge whose end-vertices we coloured differently), $|S''|=|T''|\geq 1$. Hence, we can consider $(S',T')$ as a generalized starting pair with core $(S'',T'')$. 

Our algorithm will now check if $G$ has a valid $(S'',T'',S',T')$ red-blue colouring by applying rules R1--R3 exhaustively. If we find a no-answer, then we can discard the option by Lemma~\ref{l-ll1b}. Otherwise, we found in polynomial time, again by Lemma~\ref{l-ll1b}, 
a $4$-tuple $(S,T,X,Y)$, for which the following holds:

\begin{itemize}
\item [(i)] $S''\subseteq S\subseteq X$; $S'\subseteq X$; $T''\subseteq T\subseteq Y$; $T'\subseteq Y$; and $X\cap Y=\emptyset$,\\[-10pt]
\item [(ii)] $G$ has a valid red-blue $(S'',T'',S',T')$-colouring  if and only if $G$ has valid red-blue $(S,T,X,Y)$-colouring, and\\[-10pt]
\item [(iii)] every vertex in $V\setminus (X\cup Y)$ has no neighbour in $S\cup T$; at most one neighbour in $X\setminus S$ and at most one neighbour in $Y\setminus T$.
\end{itemize}

\noindent
We now prove the following claim.

\medskip
\noindent
{\it Claim.  All connected components of $G-(X\cup Y)$ are monochromatic in every valid red-blue $(S,T,X,Y)$-colouring of $G$.}

\medskip
\noindent
We prove the claim as follows.
For a contradiction, assume $G$ has a valid red-blue $(S,T,X,Y)$-colouring, for which at least one connected component $F$ of $G-(X\cup Y)$ is not monochromatic. As  $V(G')\cup N(V(G'))\subseteq S'\cup T'$ and $S'\subseteq X$ and $T'\subseteq Y$, we find that $V(F)$ belongs to $G^*$ (recall that $G^*$ is the $P_3$-free graph obtained from $G$ after deleting the vertices of $V(G')\cup N(V(G'))$). Hence, $F$ is $P_3$-free and thus as $F$ is connected, $F$ must be a complete graph. If $F$ consists of one vertex or at least three vertices, then $F$ must be monochromatic. Hence, $F$ consists of exactly two vertices $x$ and $y$.

By statement (iii), both $x$ and $y$ have no neighbour in $S\cup T$; at most one neighbour in $X\setminus S$ and at most one neighbour in $Y\setminus T$. If $x$ and $y$ have a common neighbour, then $F$ must be monochromatic. 
If one of them, say $x$, has a neighbour in both $X\setminus S$ and a neighbour in $Y\setminus T$, then $y$ must be coloured with the same colour as $x$ and thus again $F$ must be monochromatic. 
As $G$ has minimum degree~$2$, this means that $x$ and $y$ each have exactly one neighbour in $X\cup Y$, and these two vertices in $X\cup Y$ must be different. 

Let $x'$ be the unique neighbour of $x$ in $X\cup Y$, and let $y'$ be the unique neighbour of $y$ in $X\cup Y$, so $x'\neq y'$. However, now we find that $G$ has a matching cut of size~$2$, namely the set $\{xx',yy'\}$, a contradiction. This completes the proof of the claim. \dia

\medskip
\noindent
Due to the above claim, we can now check in polynomial time, by using Lemma~\ref{l-ll2b}, whether $G$ has a valid red-blue $(S,T,X,Y)$-colouring in which every connected component in $G-(X\cup Y)$ is monochromatic. If so, then we are done, and else we discard the option.

The correctness of our algorithm follows from its description. As the total number of branches is $O(m2^{|V(H)|}n^{|V(H)|})$ and we can process each branch in polynomial time, the total running time of our algorithm is polynomial.
Hence, we have proven the theorem.
\qed
\end{proof}

\noindent
Bonsma~\cite{Bo09} proved that {\sc Matching Cut} is polynomial-time solvable for the class of {\it claw-free} graphs, that is, $K_{1,3}$-free graphs.
By combining, respectively, Corollary~\ref{c-p6} and Bonsma's result with $s$ applications of Theorem~\ref{t-h} we obtain the following result.

\begin{theorem}\label{t-p3p6}
For every integer $s\geq 0$, {\sc Matching Cut} is polynomial-time solvable for $(sP_3+P_6)$-free graphs and for $(sP_3+K_{1,3})$-free graphs.
\end{theorem}

\section{A Partial Complexity Classification for H-Free Graphs}\label{s-class}

The {\it girth} of a graph that is not a tree is the length of a shortest cycle in it.
Bonsma~\cite{Bo09} proved that {\sc Matching Cut} is \NP-complete for planar graphs of girth~$5$, and thus for $C_r$-free graphs with $r\in \{3,4\}$. Hence, {\sc Matching Cut} is \NP-complete for $H$-free graphs whenever $H$ contains a $C_3$ or $C_4$.  Le and Randerath~\cite{LR03} proved that
{\sc Matching Cut} is \NP-complete for bipartite graphs of minimum degree~$3$ and maximum degree~$4$. 
Consequently, {\sc Matching Cut} is \NP-complete for $H$-free graphs whenever $H$ contains an odd cycle.
We use a result of Moshi~\cite{Mo89} to prove the same for the case where $H$ has a not necessarily odd cycle.

Let $uv$ be an edge in a graph $G$. We replace the edge by two new vertices $w_1$ and $w_2$ and edges $uw_1$, $uw_2$, $vw_1$ and $vw_2$. We call this operation a  {\it $K_{2,2}$-replacement} and denote the resulting graph by $G_{uv}$; see also Fig.~\ref{f-moshi}. We can now state the following lemma.

\begin{figure}
\begin{center}
\scalebox{1.5}{\begin{tikzpicture}
\tikzstyle{bvertex}=[thin,circle,inner sep=0.cm, minimum size=1.7mm, fill=black, draw=black]
\tikzstyle{edge} = [thin, gray]

\def\k{0.7}
\node[bvertex, label= below:$u$](u) at (0,0){};
\node[bvertex, label= below:$v$](v) at (\k,0){};

\node[bvertex](u1) at (150:\k){};
\node[bvertex](u2) at (180:\k){};
\node[bvertex](u3) at (210:\k){};

\begin{scope}[shift= {(\k,0)}]
\node[bvertex](v1) at (45:\k){};
\node[bvertex](v2) at (15:\k){};
\node[bvertex](v3) at (-15:\k){};
\node[bvertex](v4) at (-45:\k){};
\end{scope}

\draw[thick, nicered](u)--(v);

\foreach \i in {1,...,3}{
	\draw[edge] (u) -- (u\i);
	}
	\foreach \i in {1,...,4}{
	\draw[edge] (v) -- (v\i);
	}

\begin{scope}[shift = {(5*\k,0)}]
\node[bvertex, label= below:$u$](u) at (0,0){};
\node[bvertex, label= below:$v$](v) at (1.41*\k,0){};
\node[bvertex, label= above:$w_1$](w1) at (0.705*\k,0.65*\k){};
\node[bvertex, label= below:$w_2$](w2) at (0.705*\k,-0.65*\k){};

\node[bvertex](u1) at (150:\k){};
\node[bvertex](u2) at (180:\k){};
\node[bvertex](u3) at (210:\k){};

\begin{scope}[shift= {(1.41*\k,0)}]
\node[bvertex](v1) at (45:\k){};
\node[bvertex](v2) at (15:\k){};
\node[bvertex](v3) at (-15:\k){};
\node[bvertex](v4) at (-45:\k){};
\end{scope}

%\draw[thick, OrangeRed](u)--(v);
\draw[thick, nicered](u)--(w1);
\draw[thick, nicered](u)--(w2);
\draw[thick, nicered](w1)--(v);
\draw[thick, nicered](w2)--(v);

\foreach \i in {1,...,3}{
	\draw[edge] (u) -- (u\i);
	}
	\foreach \i in {1,...,4}{
	\draw[edge] (v) -- (v\i);
	}
\end{scope}

\end{tikzpicture}}
\caption{The $K_{2,2}$-replacement applied on edge $uv$.}\label{f-moshi} 
\end{center}
\end{figure}
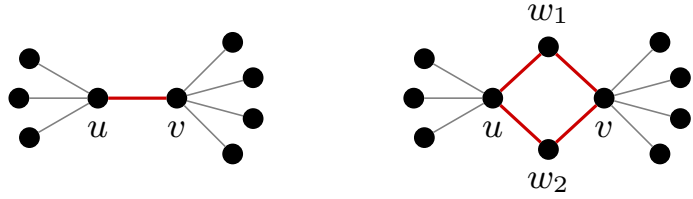

\begin{lemma}[\cite{Mo89}]\label{l-uv}
For any edge $uv$ of a graph $G$, the graph $G$ has a matching cut if and only if $G_{uv}$ has a matching cut.
\end{lemma}

\noindent
We can show the following result.

\begin{lemma}\label{l-cycle}
For every graph $H$ that is not a forest, {\sc Matching Cut} is \NP-complete for $H$-free graphs.
\end{lemma}

\begin{proof}
Let $H$ be a graph with a cycle.
If $H$ contains an induced $C_4$, then we obtain \NP-completeness as an immediate consequence of the aforementioned \NP-completeness result of Bonsma~\cite{Bo09} for planar graphs of girth~$5$. 
Now assume that $H$ is $C_4$-free. 

We reduce from {\sc Matching Cut} for general graphs. Let $G$ be a graph.
On each edge of $G$ we apply sufficiently many $K_{2,2}$-replacements such that every cycle in the resulting graph $G'$ that is not isomorphic to $C_4$ has length at least $|V(H)|+1$. As $H$ is $C_4$-free and $H$ has a cycle, this means that $G'$ is $H$-free.
 By repeated applications of Lemma~\ref{l-uv}, we find that $G$ has a matching cut if and only if $G'$ has a matching cut. \qed
\end{proof}

\begin{figure}
\begin{center}
  \begin{tikzpicture}

    \tikzstyle{vertex}=[thin,circle,inner sep=0.cm, minimum size=1.7mm, fill=black, draw=black]
\tikzstyle{hedge}=[thick, draw = gray]

    \def\k{0.7}      
          
    \node[vertex] (a1) at (0,0){};
    \node[vertex] (a2) at (1*\k,0){};
	\node[vertex] (a3) at (2*\k,0){};
    \node[vertex] (a4) at (3*\k,0){};
    \node[vertex] (a5) at (4*\k,0){};
    \node[vertex] (a6) at (5*\k,0){};
    
    \node[vertex] (b1) at (1.5*\k,\k){};
    \node[vertex] (b2) at (2.5*\k,\k){};
	\node[vertex] (b3) at (3.5*\k,\k){};

	\draw[hedge](a1)--(a2);
	\draw[hedge](a3)--(a2);
	\draw[hedge](a3)--(a4);
	\draw[hedge](a5)--(a4);
	\draw[hedge](a5)--(a6);
 
    \draw[hedge](b1)--(b2);
	\draw[hedge](b3)--(b2);

\begin{scope}[shift = {(7*\k,0)}]
    \node[vertex] (b1) at (\k,-0.5*\k){};
    \node[vertex] (b2) at (\k,0.5*\k){};
	\node[vertex] (b3) at (\k,1.5*\k){};
	
	\node[vertex] (a1) at (2*\k,0.5*\k){};
	
    \node[vertex] (c1) at (3*\k,-0.5*\k){};
    \node[vertex] (c2) at (3*\k,0.5*\k){};
	\node[vertex] (c3) at (3*\k,1.5*\k){};
	
	\draw[hedge](a1)--(c1);
	\draw[hedge](a1)--(c2);
	\draw[hedge](a1)--(c3);		
	
    \draw[hedge](b1)--(b2);
	\draw[hedge](b3)--(b2); 
\end{scope}
	
\end{tikzpicture}
\caption{The graphs $P_3+P_6$ (left) and $P_3+K_{1,3}$ (right).}\label{f-smallgraphs}
\end{center}
\end{figure}
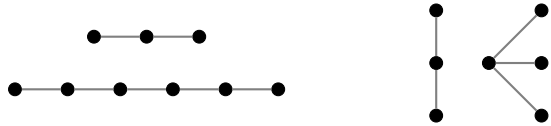

\noindent
For two graphs $G_1$ and $G_2$ we write $G_1\ssi G_2$ if $G_1$ is an induced subgraph of~$G_2$.
Let $H^*$ be the ``H''-graph, which is the graph with vertices $a_1,b_1,c_1$, $a_2,b_2,c_2$ and edges $a_ib_i$, $b_ic_i$ ($i=1,2$) and $b_1b_2$.
We can now show the following summary theorem; see also Fig.~\ref{f-smallgraphs}.

\begin{theorem}\label{t-clas}
For a graph~$H$, {\sc Matching Cut} on $H$-free graphs is 
\begin{itemize}
\item polynomial-time solvable if $H\ssi sP_3+K_{1,3}$ or $sP_3+P_6$ for some $s\geq 0$, and
\item \NP-complete if $H\si C_r$ for some $r\geq 3$, $K_{1,4}$, $P_{19}$, $4P_5$ or $H^*$.
\end{itemize}
\end{theorem}

\begin{proof}
The polynomial-time solvable cases follow from Theorem~\ref{t-p3p6}.
If $H$ has a cycle, then we apply Lemma~\ref{l-cycle}. Now suppose that $H$ has no cycle so $H$ is a forest.
First suppose that $H$ is a forest that contains a vertex of degree at least~$4$. Then $H$ contains an induced $K_{1,4}$, and thus the class of $H$-free graphs contains the class of $K_{1,4}$-free graphs.
The result then follows from the aforementioned result of Chv{\'{a}}tal~\cite{Ch84}, which in fact shows that {\sc Matching Cut} is \NP-complete even for $K_{1,4}$-free graphs, as observed by Bonsma~\cite{Bo09} and Kratch and Le~\cite{KL16}. The remaining three results are proven in~\cite{LPRb}.\qed
\end{proof}

\noindent
Every forest that is $P_r$-free for some positive constant~$r$ and that has maximum degree at most~$4$ has a constant-bounded diameter, so has constant size. Hence, Theorem~\ref{t-clas} has the following consequence.

\begin{corollary}\label{c-open}
There only exists a finite number of \emph{connected} graphs $H$ for which the computational complexity of {\sc Matching Cut} is open when restricted on $H$-free graphs.
\end{corollary}

\section{Conclusions}\label{s-con}

We gave a complexity dichotomy for {\sc Matching Cut} for graphs of bounded radius and proved 
a number of new results on the complexity of {\sc Matching Cut} for $H$-free graphs. We summarized all the known results for $H$-free graphs in Theorem~\ref{t-clas} and showed that although there still exists an infinite number of unresolved cases, the number of open cases where $H$ is a connected graph is finite.

We finish our paper with a number of open problems. Recall that Le and Le~\cite{LL19} showed that {\sc Matching Cut} for bipartite graphs of diameter at most~$d$ is polynomial-time solvable if $d\leq 3$ and \NP-complete for $d\geq 4$. Their hardness construction has radius~$4$, and we therefore pose the following open problem:

\begin{open}
Determine the complexity of {\sc Matching Cut} for bipartite graphs of radius~$3$.
\end{open}

\noindent
A standard ingredient of determining the complexity of a problem for $H$-free graphs is to first consider classes of large girth. If a problem is \NP-complete for graphs of girth at least~$g$, for every fixed integer~$g\geq 3$, then it is \NP-complete for $H$-free graphs whenever $H$ has a cycle; just take $g$ to be larger than the length of a largest cycle in $H$. However, to prove Lemma~\ref{l-cycle} we could only use the result of Bonsma~\cite{Bo09} for (planar) graphs of girth~$5$ and had to rely on the construction of Moshi~\cite{Mo89} to show hardness if $H$ has a cycle of length at least~$5$. Hence, we believe the following open problem of Le and Le~\cite{LL19} is interesting.

\begin{open}[\cite{LL19}]
Determine for every $g\geq 6$, the complexity of {\sc Matching Cut} for graphs of girth~$g$.
\end{open}

\noindent
For $K_{1,t}$-free graphs, the complexity of {\sc Matching Cut} is fully determined (see Theorem~\ref{t-clas}) with the problem becoming \NP-complete for $t\geq 4$. On the positive side,  Kratsch and Le~\cite{KL16} proved that {\sc Matching Cut} can be solved in polynomial time for $K_{1,4}$-free graphs that in addition are also $(K_{1,4}+e)$-free, where $K_{1,4}+e$ is the graph obtained from~$K_{1,4}$ by adding an edge between two of its leaves. We finish our paper with the following open problem, for which we identified some borderline cases; the {\it chair} is the graph obtained from the 
claw~$K_{1,3}$ after subdividing one of its edges exactly once.

\begin{open}
Complete the classification of {\sc Matching Cut} for $H$-free graphs; in particular what is the complexity of {\sc Matching Cut} for chair-free graphs, 
$2P_4$-free graphs and $P_7$-free graphs?
\end{open}


\begin{thebibliography}{10}

\bibitem{ACGH12}
J.~Ara{\'{u}}jo, N.~Cohen, F.~Giroire, and F.~Havet.
\newblock Good edge-labelling of graphs.
\newblock {\em Discrete Applied Mathematics}, 160:2502--2513, 2012.

\bibitem{AKK17}
N.~R. Aravind, S.~Kalyanasundaram, and A.~S. Kare.
\newblock On structural parameterizations of the matching cut problem.
\newblock {\em Proc. {COCOA} 2017, LNCS}, 10628:475--482, 2017.

\bibitem{AS21}
N.~R. Aravind and R.~Saxena.
\newblock An {F}{P}{T} algorithm for matching cut and $d$-cut.
\newblock {\em Proc. {IWOCA} 2021, LNCS}, 12757:531--543, 2021.

\bibitem{Bo09}
P.~S. Bonsma.
\newblock The complexity of the matching-cut problem for planar graphs and
  other graph classes.
\newblock {\em Journal of Graph Theory}, 62:109--126, 2009.

\bibitem{BJ08}
M.~Borowiecki and K.~Jesse{-}J{\'{o}}zefczyk.
\newblock Matching cutsets in graphs of diameter 2.
\newblock {\em Theoretical Computer Science}, 407:574--582, 2008.

\bibitem{CS16}
E.~Camby and O.~Schaudt.
\newblock A new characterization of ${P}_k$-free graphs.
\newblock {\em Algorithmica}, 75:205--217, 2016.

\bibitem{CHLLP21}
C.~Chen, S.~Hsieh, H.~Le, V.~B. Le, and S.~Peng.
\newblock Matching cut in graphs with large minimum degree.
\newblock {\em Algorithmica}, 83:1238--1255, 2021.

\bibitem{Ch84}
V.~Chv{\'{a}}tal.
\newblock Recognizing decomposable graphs.
\newblock {\em Journal of Graph Theory}, 8:51--53, 1984.

\bibitem{FP82}
A.~M. Farley and A.~Proskurowski.
\newblock Networks immune to isolated line failures.
\newblock {\em Networks}, 12:393--403, 1982.

\bibitem{Fe}
C.~Feghali.
\newblock A note on matching-cut in ${P}_t$-free graphs.
\newblock {\em Information Processing Letters}, 179:106294, 2023.

\bibitem{GPS12}
P.~A. Golovach, D.~Paulusma, and J.~Song.
\newblock Computing vertex-surjective homomorphisms to partially reflexive
  trees.
\newblock {\em Theoretical Computer Science}, 457:86--100, 2012.

\bibitem{Gr70}
R.~L. Graham.
\newblock On primitive graphs and optimal vertex assignments.
\newblock {\em Annals of the New York Academy of Sciences}, 175:170--186, 1970.

\bibitem{KKL20}
C.~Komusiewicz, D.~Kratsch, and V.~B. Le.
\newblock Matching cut: Kernelization, single-exponential time fpt, and exact
  exponential algorithms.
\newblock {\em Discrete Applied Mathematics}, 283:44--58, 2020.

\bibitem{KL16}
D.~Kratsch and V.~B. Le.
\newblock Algorithms solving the matching cut problem.
\newblock {\em Theoretical Computer Science}, 609:328--335, 2016.

\bibitem{LL19}
H.~Le and V.~B. Le.
\newblock A complexity dichotomy for matching cut in (bipartite) graphs of
  fixed diameter.
\newblock {\em Theoretical Computer Science}, 770:69--78, 2019.

\bibitem{LR03}
V.~B. Le and B.~Randerath.
\newblock On stable cutsets in line graphs.
\newblock {\em Theoretical Computer Science}, 301:463--475, 2003.

\bibitem{LPRb}
F. Lucke, D. Paulusma and B. Ries. Finding matching cuts in $H$-free graphs, Manuscript, arXiv:2207.07095.

\bibitem{MS16}
G.~B. Mertzios and P.~G. Spirakis.
\newblock Algorithms and almost tight results for $3$-{C}olorability of small
  diameter graphs.
\newblock {\em Algorithmica}, 74:385--414, 2016.

\bibitem{Mo89}
A.~M. Moshi.
\newblock Matching cutsets in graphs.
\newblock {\em Journal of Graph Theory}, 13:527--536, 1989.

\bibitem{PP01}
M.~Patrignani and M.~Pizzonia.
\newblock The complexity of the matching-cut problem.
\newblock {\em Proc. WG 2001, LNCS}, 2204:284--295, 2001.

\bibitem{HP10}
P.~van~'t Hof and D.~Paulusma.
\newblock A new characterization of ${P}_6$-free graphs.
\newblock {\em Discrete Applied Mathematics}, 158:731--740, 2010.

\end{thebibliography}
\end{document}